\documentclass[10pt,draft]{article}
\usepackage{amsmath,amsthm,amssymb,amsfonts}
\usepackage{enumerate}
\usepackage{mathrsfs}
\usepackage[cmtip,all]{xy}
\usepackage{tikz}
\usetikzlibrary{shapes.geometric}

\numberwithin{equation}{section}
\numberwithin{figure}{section}

\normalsize

\newtheoremstyle{theoremstyle}
  {10pt}      
  {5pt}       
  {\itshape}  
  {}          
  {\bfseries} 
  {:}         
  {.5em}      
  {}          

\newtheoremstyle{examplestyle}
  {10pt}      
  {5pt}       
  {}          
  {}          
  {\bfseries} 
  {:}         
  {.5em}      
  {}          

\theoremstyle{theoremstyle}
\newtheorem{theorem}{Theorem}[section]
\newtheorem*{theorem*}{Theorem}
\newtheorem{lemma}[theorem]{Lemma}
\newtheorem{proposition}[theorem]{Proposition}
\newtheorem*{proposition*}{Proposition}

\newtheorem*{corollary*}{Corollary}

\theoremstyle{examplestyle}

\newtheorem{definition}[theorem]{Definition}
\newtheorem{definition*}{Definition}
\newtheorem{remark}[theorem]{Remark}
\newtheorem{remark*}{Remark}

\let\smash=\wedge
\let\iso=\cong

\let\directsum=\oplus

\newcommand{\unit}{\mathbf{1}}
\newcommand{\SH}{\mathbf{SH}}
\newcommand{\eff}{\mathbf{eff}}
\newcommand{\veff}{\mathbf{veff}}

\newcommand{\sliced}{\mathbf{d}}

\newcommand{\holim}{\operatornamewithlimits{holim}}
\newcommand{\Ext}{{\operatorname{Ext}}}
\newcommand{\Spec}{{\operatorname{Spec}}}

\newcommand{\PP}{\mathbf{P}}
\newcommand{\MGL}{\mathbf{MGL}}

\newcommand{\MU}{\mathbf{MU}}

\newcommand{\KQ}{\mathbf{KQ}}
\newcommand{\KW}{\mathbf{KW}}

\newcommand{\KO}{\mathrm{KO}}
\newcommand{\KU}{\mathrm{KU}}
\newcommand{\ko}{\mathrm{ko}}
\newcommand{\ku}{\mathrm{ku}}

\newcommand{\KGL}{\mathbf{KGL}}
\newcommand{\kgl}{\mathbf{kgl}}
\newcommand{\kq}{\mathbf{kq}}
\newcommand{\kqq}{\f_{0}\KQ}

\newcommand{\E}{\mathsf{E}}

\newcommand{\M}{\mathbf{M}}

\newcommand{\F}{\mathsf{F}}

\newcommand{\dd}{\mathsf{d}}

\newcommand{\s}{\mathsf{s}}
\newcommand{\f}{\mathsf{f}}

\newcommand{\hyp}{\mathrm{h}}
\newcommand{\forg}{\mathrm{f}}

\newcommand{\ZZ}{{\mathbb Z}}
\newcommand{\pr}{{\mathrm{pr}}}
\newcommand{\id}{{\mathrm{id}}}

\newcommand{\A}{\mathbf{A}}
\newcommand{\G}{\mathbf{G}_{\mathfrak{m}}}

\newcommand{\Char}{\mathsf{char}}

\newcommand{\CC}{\mathbb{C}}

\newcommand{\MZ}{\mathbf{M}\mathbb{Z}}

\newcommand{\Z}{\mathbb{Z}}

\newcommand{\Sq}{\mathsf{Sq}}

\newcommand{\Sm}{\mathbf{Sm}}

\title{{\bf On very effective hermitian $K$-theory}}
\author{Alexey Ananyevskiy, Oliver R\"ondigs, Paul Arne {\O}stv{\ae}r}
\begin{document}
\maketitle
\begin{abstract}
We argue that the very effective cover of hermitian $K$-theory in the sense of motivic homotopy theory is a convenient algebro-geometric generalization of the connective real topological $K$-theory spectrum.
This means the very effective cover acquires the correct Betti realization, 
its motivic cohomology has the desired structure as a module over the motivic Steenrod algebra, 
and that its motivic Adams and slice spectral sequences are amenable to calculations.
\end{abstract}

\section{Introduction}
\label{section:introduction}
Algebraic and hermitian $K$-theory have been widely studied since the pioneering works on the Grothendieck-Riemann-Roch theorem \cite{MR0354655} and on rings with anti-involutions \cite{MR592291}.
Both theories are representable in the stable motivic homotopy category $\SH$ of a field of characteristic $\neq 2$, 
and more generally over regular noetherian base schemes of finite Krull dimension on which $2$ is invertible \cite{hornbostel.hermitian}, \cite{VoevodskyICM1998}.
Fundamental properties imply that the motivic spectra of algebraic $K$-theory $\KGL$ and hermitian $K$-theory $\KQ$ are $(2,1)$- and $(8,4)$-periodic,
respectively, 
with respect to the standard motivic spheres $S^{p,q}:=S^{p-q}\wedge\G^{q}$ for $p\geq q$.
More precisely, 
there exist Bott elements in the Grothendieck group $\pi_{2,1}\KGL\iso K_0$ and in the Grothendieck-Witt group $\pi_{8,4}\KQ\iso GW_0$ inducing motivic weak equivalences 
\begin{equation}
\label{equation:bott}
S^{2,1}\smash \KGL \xrightarrow{\iso} \KGL 
\text{ and }
S^{8,4}\smash \KQ \xrightarrow{\iso} \KQ. 
\end{equation}
Because of \eqref{equation:bott}, 
$\KGL$ and $\KQ$ are non-connective and should be thought of as ``large'' motivic spectra. 
When using $K$-theoretic invariants to inform the homotopy sheaves of the sphere $\unit$ in \cite{rso.first}, 
it is convenient to employ smaller ``connective'' versions of the motivic $K$-theory spectra.
The geometrical meaning of this notion is still not well understood.

One fascinating aspect of motivic homotopy theory is that it offers different notions of ``connectivity'' based on: 
\begin{itemize}
\item Voevodsky's slice filtration for the localizing triangulated subcategory of effective motivic spectra $\SH^\eff\subset \SH$ generated under homotopy colimits by motivic $\PP^{1}$-suspension spectra of smooth schemes 
\cite[\S2]{voevodsky.open}.
\item Morel's homotopy $t$-structure \cite[\S6.2]{morel.connectivity} characterized over perfect fields by the vanishing of homotopy sheaves, 
and its extension to general base schemes in \cite[\S2.1]{hoyois}.
\item Spitzweck-{\O}stv{\ae}r's very effective slice filtration for the full subcategory $\SH^\veff\subset \SH$ generated under homotopy colimits and extensions by motivic $\PP^{1}$-suspension spectra of smooth schemes 
\cite[Definition 5.5]{so.twisted}.
\end{itemize}

The essential difference between the effective and the very effective slice filtrations is that the former records slices with respect to the multiplicative group scheme $\G$ and the latter with respect to 
the projective line $\PP^{1}$.
By construction $\SH^\veff$ is a subcategory of $\SH^\eff$ closed under the tensor product, 
but it is not closed under simplicial desuspension and hence not a triangulated subcategory of $\SH$.
The sphere $\unit$, 
algebraic cobordism $\MGL$, 
and quotients thereof such as motivic Moore spectra and the effective cover of $\KGL$ are examples of very effective motivic spectra.
Both of the slice filtrations interact well with $A_{\infty}$- and $E_{\infty}$-structures \cite{grso}, 
but only the very effective one maps to (the even part of) the topological Postnikov filtration under Betti realization \cite[\S3.3]{grso}.
Over perfect fields $\SH^\veff$ is the nonnegative part of the $t$-structure on $\SH^\eff$,
and the identification of the very effective slices of $\KQ$ up to extensions in \cite[Theorem 16]{bachmann.very} makes a strong case for further investigations of the very effective slice filtration.

In this paper we argue that the very effective cover $\kq$ of hermitian $K$-theory $\KQ$ is a convenient algebro-geometric generalization of the connective cover $\ko$ of real topological $K$-theory $\KO$.
The question of whether there is a motivic spectrum with similar properties as $\ko$ was first addressed in \cite{isaksen-shkembi} and \cite[Conjecture 5.8]{hill} over the fields of complex and real numbers,
respectively.

Section~\ref{sec:conn-conn-k} begins with some preliminary results on $\KGL$.
We identify the effective and very effective covers of $\KGL$ over perfect fields of characteristic $\neq 2$, 
and similarly for $\KGL/2$ and base schemes over $\Spec(\Z[\frac{1}{2}])$.
Proposition \ref{prop:wood} relates the very effective covers $\kq$ of $\KQ$ and $\kgl$ of $\KGL$ to the motivic Hopf map $\eta\colon\A^{2}\smallsetminus\{0\}\to\PP^{1}$ 
via the cofiber sequence
\begin{equation}
\label{equation:woodveryeffectivecovers}
\Sigma^{1,1}\kq 
\xrightarrow{\eta} \kq 
\rightarrow 
\kgl.
\end{equation}
Over the complex numbers, \eqref{equation:woodveryeffectivecovers} has been obtained independently by Bachmann.
The same result holds for $\eta$, $\KQ$, $\KGL$, and base schemes over $\Spec(\Z[\frac{1}{2}])$ by \cite[Theorem 3.4]{ro.hermitian}, 
but it is plainly false for the effective covers of $\KQ$ and $\KGL$ by the proof of \cite[Corollary 5.1]{rso.hlp}.
By using \eqref{equation:woodveryeffectivecovers} we identify the Betti realization of $\kq$ with $\ko$ and calculate the mod-$2$ motivic cohomology $\M\Z/2^{\star}\kq$ as 
$\mathcal{A}^{\star}//\mathcal{A}^{\star}(1)$;
the quotient of the mod-$2$ motivic Steenord algebra $\mathcal{A}^{\star}$ by the augmentation ideal of the $\M\Z/2^{\star}$-subalgebra generated by $\Sq^{1}$ and $\Sq^{2}$ 
\cite{hko.positivecharacteristic}, \cite{Voevodsky.reduced}.
By dualizing, 
the mod-$2$ motivic homology $\M\Z/2_{\star}\kq$ identifies with $\mathcal{A}_{\star}\Box_{\mathcal{A}_{\star}(1)}\M\Z/2_{\star}$ as an $\mathcal{A}_{\star}$-comodule algebra, 
and by change-of-rings the $\M\Z/2$-based Adams spectral sequence for $\kq$ takes the form
\begin{equation}
\label{equation:masskql}
\Ext^{\ast,\star}_{\mathcal{A}_{\star}(1)}(\M\Z/2_{\star},\M\Z/2_{\star})
\Rightarrow
\kq_{\star}{}^{{\kern -.5pt}\wedge}_{2,\eta}.
\end{equation}
As indicated in the notation, the filtered target groups of the spectral sequence \eqref{equation:masskql} are all $(2,\eta)$-completed.
The $\Ext$-algebra over $\mathcal{A}_{\star}(1)$ appearing in \eqref{equation:masskql} is accessible via homological algebra.
For explicit calculations with \eqref{equation:masskql}  we refer to \cite{hill} and \cite{isaksen-shkembi}.


Section \ref{sec:slice-computations} is concerned with slice calculations.
The negative slices of $\kq$ are evidently zero because $\kq$ is an effective motivic spectrum.
Over a perfect field of characteristic $\neq 2$ and  $i\geq 0$, we show in Theorem~\ref{thm:slices-kq} the calculation 
\begin{equation}
\label{equation:sliceskq}
\s_{q}\kq 
=
\begin{cases}
\Sigma^{2n,2n}\MZ/2 \vee \Sigma^{2n+2,2n}\MZ/2\vee \dotsm \vee \Sigma^{4n-2,2n}\MZ/2  \vee \Sigma^{4n,2n}\MZ  & q=2n, \\
\Sigma^{2n+1,2n+1}\MZ/2 \vee \Sigma^{2n+3,2n+1}\MZ/2 \vee \dotsm \vee \Sigma^{4n+1,2n+1}\MZ/2 & q=2n+1.
\end{cases}
\end{equation}
The slices of $\kq$ are considerable ``smaller'' than those of $\KQ$ \cite{ro.hermitian}.  
This is a helpful fact which is used in the calculation of the first stable homotopy groups of motivic spheres \cite{rso.first}. 

An immediate consequence of \eqref{equation:sliceskq} is the explicit form of the slice spectral sequence given by mod-$2$ motivic cohomology groups $h^{\star}$ and integral motivic cohomology groups $H^{\star}$
\begin{equation}
\label{equation:ssskql}
\pi_{p,w}\s_{q}\kq
=
\begin{cases}
h^{2n-p,2n-w}
\oplus\cdots\oplus h^{4n-2-p,2n-w}\oplus H^{4n-p,2n-w} & q=2n \\
h^{2n+1-p,2n+1-w}
\oplus\cdots\oplus h^{4n+1-p,2n+1-w} & q=2n+1
\end{cases}
\Rightarrow
\kq_{p,w}.
\end{equation}
In Theorem \ref{thm:diff-kq} we identify the differentials in \eqref{equation:ssskql} in terms of motivic Steenrod operations.
We also calculate the slices and the slice differentials for the $\eta$-inversion of $\kq$.

In Section \ref{sec:homotopy-computations} we identify the $0$-line of $\kq$ with the Milnor-Witt $K$-theory over fields of characteristic not $2$, 
and determine the associated graded for the groups on the $1$-line of $\kq$.
\vspace{0.01in}

Throughout the paper we employ the following assumptions and notations.
\vspace{0.05in}

\begin{tabular}{l|l}
$F$, $S$ & perfect field, finite dimensional separated noetherian scheme \\
$\Sm_{S}$ & smooth schemes of finite type over $S$ \\
$S^{s,t}$, $\Omega^{s,t}$, $\Sigma^{s,t}$ & motivic $(s,t)$-sphere, $(s,t)$-loop space, $(s,t)$-suspension  \\
$\SH$, $\SH^{\eff}$ & motivic and effective motivic stable homotopy categories of $S$\\ 
$\E$, $\unit=S^{0,0}$ & generic motivic spectrum, the motivic sphere spectrum  \\
$\Lambda$, $\mathbf{M}A$ & ring, motivic Eilenberg-MacLane spectra of a $\Lambda$-module $A$ \\
$\KGL$, $\KQ$, $\KW$ & algebraic and hermitian $K$-theory, Witt-theory \\
$\f_{q}$, $\tilde{\f}_{q}$, $\s_{q}$ & $q$th effective cover, very effective cover, and slice 
\end{tabular}
\vspace{0.05in}
\noindent

In all results concerning $\KQ$ and $\kq$ we assume that $2$ is invertible on the base scheme $S$, 
as for $\Spec(\ZZ[\tfrac{1}{2}])$, 
and following \cite{spitzweckmz} we impose the condition that 
\begin{equation}
\label{baseschemeassumption}
S\text{ is essentially smooth over a Dedekind domain.}
\end{equation}
Applications will mostly concern perfect fields of characteristic $\neq 2$.

\section{Connecting connective $K$-theories}
\label{sec:conn-conn-k}
\begin{definition}
\label{def:very-effective-cover}
Following \cite[\S5]{so.twisted} we let 
$\kq\to \KQ$ denote the very effective cover of the hermitian $K$-theory spectrum $\KQ$ of quadratic forms \cite{hornbostel.hermitian} 
and let $\kgl\to \KGL$ denote the very effective cover of the algebraic $K$-theory spectrum $\KGL$ \cite{VoevodskyICM1998}. 
\end{definition}

\begin{remark}
\label{rem:very-perfect-field}
Following \cite{bachmann.very}, \cite{so.twisted}, and working over $F$ we can identify the very effective cover $\tilde{\f}_{q}\E$ of $\E$ with $\f_0(\E_{\geq 0})$, 
the effective cover of the connective cover $\E_{\geq 0}$ of $\E$ with respect to the homotopy $t$-structure on $\SH$ \cite{morel.connectivity}. 
\end{remark}

\begin{lemma}
\label{lem:kgl-complex-realization}
If $F$ admits a complex embedding, 
the Betti realization of $\kgl$ coincides with the connective cover $\ku$ of the complex topological $K$-theory spectrum $\KU$ in the topological stable homotopy category.
\end{lemma}
\begin{proof}
Recall from \cite[Proposition 5.12]{so.twisted} that $\kgl$ is a homotopy quotient of $\MGL$ under the orientation or Todd genus map, 
and similarly but easier that $\ku$ is a homotopy quotient of $\MU$.
The Betti realization functor \cite[Appendix A]{PPR} preserves homotopy colimits, and sends $\MGL$ to $\MU$.
\end{proof}

\begin{lemma}
\label{lem:kgl-effective-perfect}
Over $F$ the effective and very effective covers of $\KGL$ coincide in $\SH$. 
\end{lemma}
\begin{proof}
When $\Char(F)=0$ this is shown in \cite[Corollary 5.13]{so.twisted} by writing the effective cover of $\KGL$ as a homotopy quotient of $\MGL$ 
(the latter is very effective over any base scheme \cite[Theorem 5.7]{so.twisted}). 
If $\Char(F)>0$ we follow the proof of \cite[Theorem 16]{bachmann.very} where the effective cover $\f_{0}\KGL \to \KGL$ is shown to be connective.
For $t\geq 0$ the presheaf on $\Sm_F$
\[ 
X\mapsto [\Sigma^{s,t}X_+,\f_{0}\KGL] = K_{s-2t}(X) 
\]
is zero for $s<2t$, 
e.g., 
for $s-t<0$ (this holds if $X$ is regular, hence over any regular base scheme $S$).
The case $t=0$ implies by \cite[Proposition 4]{bachmann.very} that $\f_{0}\KGL$ is connective, and by \cite[Lemma 10]{bachmann.very} that $\f_{0}\KGL$ is the very effective cover.
\end{proof}
\begin{remark}
Lemma \ref{lem:kgl-effective-perfect} holds more generally for motivic Landweber exact spectra in the sense of \cite{NSO}.
\end{remark}

Over a noetherian scheme $S$ of finite Krull dimension $d$, 
the presheaf on $\Sm_S$
\[ 
X\mapsto [\Sigma^{s,t}X_+,\KGL] = KH_{s-2t}(X) 
\]
is zero for $s-2t <-d$ by \cite{kerz-strunk}, since
$\KGL$ represents homotopy $K$-theory over $S$ \cite{cisinski.homotopyk}.
Thus for $t\geq q$, the presheaf 
\[ 
X\mapsto [\Sigma^{s,t}X_+,\f_{q}\KGL] = K_{s-2t}(X) 
\]
is zero for $s-t+d<q$, and $\f_{q}\KGL$ is $q$-connected in the sense of \cite[Definition 3.16]{rso.first}.
If the very effective slice filtration coincides with the combination of the homotopy $t$-structure and the effective slice filtration over $S$, 
then $\f_{0}\KGL$ is the very effective cover, 
i.e., 
the effective and very effective slices of $\KGL$ agree. 
We can argue differently for $\KGL/2$ when $2$ is invertible (this proof can also be adapted to motivic Landweber exact spectra).

\begin{lemma}
\label{lem:kglmod2-effective}
Over a base scheme $S$ as in \eqref{baseschemeassumption} on which $2$ is invertible,
the effective and very effective covers of $\KGL/2$ coincide in $\SH$. 
\end{lemma}
\begin{proof}
We claim $\KGL/2$ affords the description as a homotopy quotient of $\MGL/2$ for the generators of the Lazard ring $x_i\in\pi_{2i,i}\MGL$. 
Since $\MGL$ is effective the orientation map for $\KGL$ factors through 
$$
{\phi}\colon\MGL\to \f_{0}\KGL.
$$ 
For $i\geq 2$ we have $\pi_{2i,i}\phi(x_i)=0$, so that ${\phi}$ admits a factorization 
$$
\MGL/(x_2,x_3,\dotsc) \to \f_{0}\KGL.
$$ 
We claim there is a canonically induced motivic weak equivalence 
\[ 
\psi\colon \MGL/(2,x_2,x_3,\dotsc) \overset{\iso}{\to} \f_{0}\KGL/2.
\] 
The map $\psi$ yields an isomorphism on slices by \cite[Theorem 11.3]{spitzweckmz} and an appropriate adaption of \cite[Proposition 5.4]{spitzweck.relations}.
We show that $\holim_{q\to \infty} \f_{q} \psi$ is a map between contractible motivic spectra,
i.e., 
$\psi$ is a map between slice complete spectra. 
For $\KGL/2$ this follows by the argument prior to Lemma~\ref{lem:kglmod2-effective}:
By \cite{kerz-strunk} we know $\f_{q}\KGL$ is $q$-connected in the sense of \cite[Definition 3.16]{rso.first}.
Thus $\holim_{q\to \infty} \f_{q}\KGL \iso \ast$, and likewise for $\f_{0}\KGL/2$. 
The contractibility of $\holim_{q\to \infty}\f_{q}\MGL/(2,x_2,x_3,\dotsc)$ follows from the description of the covers $\f_{q}\MGL$ in the proof of \cite[Theorem 4.6]{spitzweck.relations}. 
To conclude for $\psi$ we use that slices detect motivic weak equivalences between slice complete motivic spectra, 
cf.~\cite[\S8.3]{hoyois}.
Recall that $\MGL$ is a very effective motivic spectrum \cite[Theorem 5.7]{so.twisted}.
The lemma follows from the canonically induced motivic weak equivalences in the commutative diagram
\[ 
\xymatrix{ 
\tilde{\f}_{0}\MGL/(2,x_2,x_3,\dotsc) \ar[r]^-{\iso} \ar[d]_-{\iso} &
\tilde{\f}_{0}\f_{0}\KGL/2 \ar[r]^-{\iso} \ar[d] &
\tilde{\f}_{0}\KGL/2 \ar[d]  \\
{\f}_{0}\MGL/(2,x_2,x_3,\dotsc) \ar[r]^-{\iso} &
\f_{0}\f_{0}\KGL/2 \ar[r]^-{\iso} &
\f_{0}\KGL/2.
}
\]
\end{proof}

The Bott element $\PP^1\to \KGL$ lifts canonically to a map $\beta\colon \PP^1\to \kgl$ because $\PP^1$ is very effective. 
Let $\gamma$ denote the canonical composite
\[ 
\kgl \to \f_{0}(\KGL)\to \s_{0}\KGL.
\]
\begin{proposition}
\label{prop:kgl-bott}
Over $F$ multiplication with the Bott element induces the cofiber sequence
\[ 
\Sigma^{2,1}\kgl 
\xrightarrow{\beta} 
\kgl 
\xrightarrow{\gamma} 
\MZ 
\xrightarrow{\delta} \Sigma^{3,1}\kgl. 
\]
\end{proposition}
\begin{proof}
By Lemma \ref{lem:kgl-effective-perfect} we have $\f_{0}(\KGL_{\ge 0})\iso\kgl$ and by $(2,1)$-periodicity $\f_{-1}(\KGL_{\ge -1})\iso\Sigma^{-2,-1}\kgl$.
Our claim follows from the commutative diagram
\[ 
\xymatrix{ 
\Sigma^{2,1}\f_0(\KGL_{\geq 0}) \ar[r]^-\iso \ar[d]_-\beta &
\f_{1} (\Sigma^{2,1} \KGL_{\geq 0}) \ar[r]^-\iso & 
\f_{1} ((\Sigma^{2,1} \KGL)_{\geq 1})\ar[r]^-\iso  & 
\f_{1} (\KGL_{\geq 1}) \ar[d]^-{\beta^\prime} \\
\f_0(\KGL_{\geq 0})  \ar[rrr]^\id &&& \f_0(\KGL_{\geq 0}) }
\]
and the cofiber sequence for the very effective zero slice of $\KGL$ \cite[Lemma 7]{bachmann.very}
\[ 
\f_{1}(\KGL_{\geq 1}) \xrightarrow{\beta^\prime} \f_0(\KGL_{\geq 0}) \to \MZ,
\]
which coincides with the usual zero slice $\s_{0}\KGL\iso\MZ$ computed in \cite{levine.coniveau}, \cite{Voevodsky:motivicss}.
\end{proof}

\begin{proposition}
\label{prop:kglmod2-bott}
Over a base scheme $S$ as in \eqref{baseschemeassumption} on which $2$ is invertible, 
multiplication with the Bott element induces the cofiber sequence
\[ 
\Sigma^{2,1}\kgl/2 
\xrightarrow{\beta} 
\kgl/2 \xrightarrow{\gamma} 
\MZ/2 \xrightarrow{\delta} 
\Sigma^{3,1}\kgl/2. 
\]
\end{proposition}
\begin{proof}
This follows from Lemma~\ref{lem:kglmod2-effective}.
\end{proof}


\begin{lemma}
\label{lem:kgl-bott}
If $2$ is invertible on a base scheme $S$ as in \eqref{baseschemeassumption}, 
then the composite
\[ 
\MZ\smash \MZ/2 \xrightarrow{\delta\smash \MZ/2} 
\Sigma^{3,1}\kgl\smash \MZ/2 \xrightarrow{\gamma \smash \MZ/2} 
\Sigma^{3,1}\MZ\smash \MZ/2 
\]
is given by multiplication with the first Milnor operation 
$$
\mathsf{Q}_{1}=\Sq^{1}\Sq^{2}+\Sq^{2}\Sq^{1}\colon \MZ/2\to \Sigma^{3,1}\MZ/2.
$$
\end{lemma}
\begin{proof}
The proof of Lemma~\ref{lem:kglmod2-effective} shows $\KGL$ and $\f_{0}\KGL/2$ are invariant under base change, 
being homotopy quotients of $\MGL$. 
The same holds for $\MZ$ and the motivic Steenrod algebra by \cite[Section 10, Theorem 11.24]{spitzweckmz}.
Hence we may assume $S=\Spec(\ZZ[\tfrac{1}{2}])$, 
and reduce to $\CC$ by rigidity of the motivic Steenrod algebra \cite[Theorem 11.24]{spitzweckmz}.
Over $\CC$ our claim follows from Lemma~\ref{lem:kgl-complex-realization} and the corresponding topological result.
\end{proof}


\begin{remark}
Following \cite[Theorem 5.4]{isaksen-shkembi}, 
Lemma \ref{lem:kgl-bott} shows the mod-$2$ motivic cohomology $\M\Z/2^{\star}\kgl$ is the quotient of the mod-$2$ motivic Steenord algebra $\mathcal{A}^{\star}$ by the augmentation ideal of the 
$\M\Z/2^{\star}$-subalgebra generated by $\mathsf{Q}_{0}=\Sq^{1}$ and $\mathsf{Q}_{1}$. 
\end{remark}

\begin{proposition}
\label{prop:wood}
Over a field of characteristic $\neq 2$, 
multiplication with the Hopf map $\eta$ induces a cofiber sequence
\begin{equation}
\label{equation:wood}
\Sigma^{1,1}\kq \xrightarrow{\eta} \kq 
\xrightarrow{\forg} 
\kgl \xrightarrow{\hyp} \Sigma^{2,1}\kq. 
\end{equation}
Here $\forg$ and $\hyp$ are functorially induced by the forgetful and hyperbolic maps between algebraic and hermitian $K$-theory,
respectively.
\end{proposition}
\begin{proof}
Consider the fiber $\F$ of the naturally induced forgetful map $\forg_{\geq 0}\colon \KQ_{\geq 0} \to \KGL_{\geq 0}$.
Since $\f_0$ is a triangulated functor, $\f_0(\F)$ is the fiber of $\forg:=\f_0(\forg_{\geq 0})$. 
The composite map
$$
\Sigma^{1,1}\kq \xrightarrow{\eta} \kq \xrightarrow{\forg} \kgl 
$$
is trivial because the first negative algebraic $K$-group $\pi_{1,1}\kgl=\pi_{1,1}\KGL=K_{-1}$ vanishes over regular schemes. 
We show there is an induced motivic weak equivalence $\Sigma^{1,1}\kq \to \f_0(\F)$ of effective motivic spectra by checking the map of homotopy sheaves $\pi_{s,t}$ is an isomorphism for $t\geq 0$.
This follows if \eqref{equation:wood} induces a long exact sequence of sheaves for $t\geq 0$
\begin{equation} 
\label{eq:wood}
\xymatrix{
\dotsm \ar[r] & \pi_{s,t} \Sigma^{1,1}\kq \ar[r]^-\eta \ar[d]_\iso &
\pi_{s,t}\kq \ar[r]^-\forg \ar[d]^= & \pi_{s,t}\kgl \ar[r]^-\hyp \ar[d]^-= & 
\dotsm\\
\dotsm \ar[r] & \pi_{s-1,t-1}\kq \ar[r]^-\eta &
\pi_{s,t}\kq \ar[r]^-\forg & \pi_{s,t}\kgl \ar[r]^-\hyp & 
\dotsm.}
\end{equation}
By construction, 
(\ref{eq:wood}) is exact for $t\geq 1$ and $s\geq t$ since in the said range it coincides with the long exact sequence
\[ 
\dotsm \to 
\pi_{s,t} \Sigma^{1,1}\KQ \xrightarrow{\eta}
\pi_{s,t}\KQ \xrightarrow{\forg} 
\pi_{s,t}\KGL \xrightarrow{\hyp} \dotsm 
\]
induced by the Wood cofiber sequence for $\eta$, $\KQ$, and $\KGL$ \cite[Theorem 3.4]{ro.hermitian}. 
  
If $t\geq 1$ and $s=t$,
$\pi_{t,t}(\forg)\colon \pi_{t,t}\kq \to \pi_{t,t}\kgl$ is surjective since its target is trivial.
Thus (\ref{eq:wood}) is exact for $t\geq 1$ and all $s$;
recall that $\pi_{s,t}\kq = \pi_{s,t}\kgl = 0$ for all $s<t$.
  
It remains to consider the case $t=0$.
By \cite[Theorem 16]{bachmann.very} the composite
\[ 
\f_0(\KQ_{\geq 0}) \to \f_{-1}(\KQ_{\geq 0}) \to \f_{-1}(\KQ_{\geq -1}) 
\]
is an equivalence.
The canonical map $\KQ_{\geq 0} \to \KQ_{\geq -1}$ is an isomorphism on homotopy sheaves $\pi_{s,t}$ for all $t\geq -1$ and all $s$. 
When $s<t-1$ and $s\geq t$ this follows by construction. 
The case $s=t-1$ holds since $\pi_{t-1,t}\KQ=0$ for all $t\geq -1$. 
More precisely,
the vanishing for $t\geq 0$ is implied by comparison with Witt theory because $\pi_{t-1,t}\KW = 0$ for all $t$. 
The case $t=-1$ follows from the long exact sequence
\[ 
\dotsm \to \pi_{0,0}\KGL \xrightarrow{0}\pi_{-2,-1}\KQ \xrightarrow{\eta} \pi_{-1,0} \KQ \xrightarrow{\forg} \pi_{-1,0}\KGL = 0, 
\]
and surjectivity of the rank map 
$\forg\colon \pi_{0,0}\kq \to \pi_{0,0}\kgl$.
It follows that there is a canonical motivic weak equivalence 
\[ 
\f_0(\KQ_{\geq 0}) \overset{\iso}{\to} \f_{-1}(\KQ_{\geq 0}), 
\]
which implies exactness of (\ref{eq:wood}) for $t=0$.
\end{proof}




\begin{lemma}
\label{lem:kq-wood}
If $2$ is invertible on a base scheme $S$ as in \eqref{baseschemeassumption},
then the composite
\[ 
\kgl\smash \MZ/2 \xrightarrow{\hyp\smash \MZ/2} 
\Sigma^{2,1}\kq\smash \MZ/2 \xrightarrow{\forg \smash \MZ/2} 
\Sigma^{2,1}\kgl\smash \MZ/2
\]
is given by multiplication with $\Sq^2\colon \MZ/2\to \Sigma^{2,1}\MZ/2$.
\end{lemma}
\begin{proof}
As in the proof of Lemma~\ref{lem:kgl-bott} it suffices to work over $\Spec(\ZZ[\tfrac{1}{2}])$, and hence over $\CC$.
The result follows from Lemma \ref{lem:kq-complex-realization} and the corresponding topological statement.
\end{proof}

\begin{lemma}
\label{lem:kq-complex-realization}
If $F$ admits a complex embedding, 
the Betti realization of $\kq$ coincides with the connective cover $\ko$ of the real topological $K$-theory spectrum $\KO$ in the topological stable homotopy category.
\end{lemma}
\begin{proof}
This follows since the Betti realization sends $\KQ$ to $\KO$, 
$\kgl$ to $\ku$ by Lemma~\ref{lem:kgl-complex-realization},
and preserves the Wood cofiber sequence.
\end{proof}

\begin{remark}
As in \cite[Theorem 5.11]{isaksen-shkembi}, 
Lemma \ref{lem:kgl-bott} identifies $\M\Z/2^{\star}\kq$ with the quotient of the mod-$2$ motivic Steenord algebra $\mathcal{A}^{\star}$ by the augmentation ideal of the 
$\M\Z/2^{\star}$-subalgebra generated by $\Sq^{1}$ and $\Sq^{2}$, 
and the homotopy of $\kq\smash \MZ/2$ as a comodule over the dual motivic Steenrod algebra recorded by the $\M\Z/2$-based Adams spectral sequence for $\kq$ \eqref{equation:masskql}.
\end{remark}

Next we observe that $\kgl$ differs from the cover of algebraic $K$-theory introduced in \cite{isaksen-shkembi}.
By the cofiber sequence
\[ 
\kgl=\f_0(\KGL_{\geq 0}) \to \f_{-1}(\KGL_{\geq -1}) \to \s_{-1}\KGL = \Sigma^{-2,-1}\MZ, 
\]
we obtain a long exact sequence and an isomorphism
\begin{equation}
\label{equation:lespi-1-1}
\dotsm 
\to \pi_{0,-1} \s_{-1}\KGL
\to \pi_{-1,-1}\kgl
\overset{\iso}{\to} 
\pi_{-1,-1}\f_{-1}(\KGL_{\geq -1}) 
\to \pi_{-1,-1}\s_{-1}\KGL 
\to \dotsm.
\end{equation}
The outer terms in \eqref{equation:lespi-1-1} are trivial.
Since $\pi_{-1,-1}\f_{-1}(\KGL_{\geq -1}) \iso \pi_{-1,-1}\KGL$ it follows that $\pi_{-1,-1}\kgl \iso K_{1}(F)\iso F^{\times}$.
Over the complex numbers, 
this calculation distinguishes $\kgl$ from the ($2$-complete) positive cellular cover of $\KGL$ in \cite{isaksen-shkembi} because $\pi_{-1,-1}$ of the latter is trivial by construction.

Finally, 
we remark that $\kq$ does not coincide with the effective cover $\kqq$ featuring in the solution of the homotopy limit problem for the $C_{2}$-action on $\kgl$ in \cite{rso.hlp}.

\section{Slice computations}
\label{sec:slice-computations}

We shall identify the slices of $\kq$ similarly to the slices of $\KQ$ in \cite{ro.hermitian}. 
The crucial ingredients are the Wood cofiber sequence \eqref{prop:wood} and the slices of connective algebraic $K$-theory $\kgl$.
\begin{theorem}
\label{thm:slices-kgl}
Over $F$ the canonical map $\kgl \to \KGL$ induces an isomorphism on all nonnegative slices.
The negative slices of $\kgl$ are zero.
\end{theorem}
\begin{proof}
Since $\kgl=\f_{0}\KGL$ by Lemma~\ref{lem:kgl-effective-perfect}, 
this follows by construction.
\end{proof}

Identifying the slices of $\kq$ is more involved because $\kq\neq\f_{0}\KQ$. 

\begin{theorem}
\label{thm:slices-kq}
When $\Char(F)\neq 2$ the nonnegative slices of $\kq$ are given as 
\begin{equation*}
\s_{q}\kq 
=
\begin{cases}
\Sigma^{2n,2n}\MZ/2 \vee \Sigma^{2n+2,2n}\MZ/2\vee \dotsm \vee \Sigma^{4n-2,2n}\MZ/2  \vee \Sigma^{4n,2n}\MZ  & q=2n, \\
\Sigma^{2n+1,2n+1}\MZ/2 \vee \Sigma^{2n+3,2n+1}\MZ/2 \vee \dotsm \vee \Sigma^{4n+1,2n+1}\MZ/2 & q=2n+1.
\end{cases}
\end{equation*}
The negative slices of $\kq$ are zero.
Moreover, 
the canonical map $\kq \to \KQ$ induces a natural inclusion on slices, and respects the multiplicative structure.
\end{theorem}
\begin{proof}
Since $\kq = \f_{0}(\KQ_{\geq 0})$ is (very) effective, its negative slices are zero. 
Applying the slice functor to \eqref{prop:wood} yields a cofiber sequence. 
The natural isomorphism $\s_{q}\circ \Sigma^{1,1} \iso\Sigma^{1,1}\circ \s_{q-1}$ of \cite[Lemma 2.1]{ro.hermitian} shows the forgetful map $\forg\colon \kq \to \kgl$ induces an isomorphism on zero slices
\[ 
\s_0\kq \xrightarrow{\iso} \s_0 \kgl, 
\]
and likewise for the unit map $\unit \to \kq$.

For the $1$-slices there is a cofiber sequence
\[ 
\Sigma^{1,1}\s_0\kq = 
\Sigma^{1,1}\MZ \xrightarrow{\eta} \s_1\kq\xrightarrow{\s_1\forg} \s_1\kgl = 
\Sigma^{2,1} \MZ \xrightarrow{\s_1\hyp} 
\Sigma^{2,1}\s_0\kq = 
\Sigma^{2,1} \MZ.
\]
Here $\s_1\hyp$ can be identified with an integer $n\in\ZZ$. 
Comparison with the hyperbolic map $\KGL \to \KQ$ in \cite[\S4.3]{ro.hermitian} shows that $n=2$, 
so that $\s_1\kq = \Sigma^{1,1}\MZ/2$.

For the $2$-slices there is a cofiber sequence
\[ 
\Sigma^{1,1}\s_1\kq =
\Sigma^{2,2}\MZ/2 \xrightarrow{\eta} \s_2\kq \xrightarrow{\s_2\forg} 
\s_2\kgl = 
\Sigma^{4,2}\MZ \xrightarrow{\s_2\hyp} 
\Sigma^{2,1}\s_1\kq = 
\Sigma^{3,2} \MZ/2.
\]
Hence $\s_2\hyp=0$, 
the cofiber sequence splits, 
and we get $\s_{2}\kq = \Sigma^{2,2}\MZ/2 \vee \Sigma^{4,2}\MZ$. 
Moreover,
$\s_{2}\forg$ is the projection map onto $\Sigma^{4,2}\MZ$.

For the $3$-slices there is a cofiber sequence
\[ 
\Sigma^{1,1}\s_2\kq =
\Sigma^{3,3}\MZ/2 \vee \Sigma^{5,3}\MZ\xrightarrow{\eta} 
\s_3\kq \xrightarrow{\s_3\forg} 
\s_3\kgl = 
\Sigma^{6,3} \MZ \xrightarrow{\s_3\hyp} 
\Sigma^{2,1}\s_2\kq.
\]
Here $\s_3\hyp$ maps trivially to $\Sigma^{4,3}\MZ/2$,
while the component of $\s_3\hyp$ mapping to $\Sigma^{6,3}\MZ$ can be identified with an integer $n\in\ZZ$.
We deduce $n=2$ by comparison with the hyperbolic map $\KGL \to \KQ$ in \cite[\S4.3]{ro.hermitian}. 
Hence we obtain $\s_{3}\kq\iso \Sigma^{3,3}\MZ/2\vee \Sigma^{5,3}\MZ/2$. 

Iterating these arguments produces the claimed calculation. 
\end{proof} 

\begin{remark}
Contrary to the calculation of the slices of $\KQ$ in \cite{ro.hermitian} there is no ``mysterious summand'' appearing in Theorem \ref{thm:slices-kq} thanks to the connectivity of $\kq$.
Each slice of $\kq$ is a finite sum of motivic Eilenberg-MacLane spectra for the groups $\ZZ$ and $\ZZ/2$.
The odd slices of $\kq$ are cellular of finite type for every $F$ \cite[\S3.3]{rso.first}, and likewise for all the slices when $\Char(F)=0$.
\end{remark}

The multiplicative structure on the graded slices $\s_{\ast}\kq$ can be identified similarly to $\s_{\ast}\KQ$ as in~\cite[Theorem 3.3]{ro.mult-slices-hermitian}. 
In more details, 
there is a motivic weak equivalence
\[ 
\s_\ast\kq
\iso 
\M\ZZ[\eta,\sqrt{\alpha}]/(2\eta=0,\eta^2\xrightarrow{\delta}\sqrt{\alpha}). 
\]
Here $\eta$ has bidegree $(1,1)$ and $\sqrt{\alpha}$ is a class of bidegree $(4,2)$ arising from the $(8,4)$-periodicity operator on $\KQ$ mentioned in the introduction.
Moreover, 
the action of the Hopf map $\eta$ on the slices of $\kq$ can be read off from the proof of Theorem~\ref{thm:slices-kq}, 
giving us the next result.

\begin{theorem}
\label{thm:slices-kw}
When $\Char(F)\neq 2$ the slices of $\kq[\frac{1}{\eta}] = \KW_{\geq 0}$ are given by
\[ 
\s_{q}(\KW_{\geq 0}) = 
\Sigma^{q,q} \Bigl( \MZ/2 \vee \Sigma^{2,0} \MZ/2 \vee \Sigma^{4,0} \MZ/2 \vee \dotsm \Bigr), 
\]
and
\[ 
\s_{\ast}(\KW_{\geq 0})
\iso 
\M\ZZ[{\eta},\sqrt{\alpha}]/(2\eta=2\sqrt{\alpha}=0,{\eta}^2\xrightarrow{\Sq^1}\sqrt{\alpha}).
\]
The canonical map $\KW_{\geq 0}\to \KW$ induces the natural inclusion on slices,
and respects the multiplicative structure.
\end{theorem}

Let $\sliced_{1}^{\kq}(q)\colon \s_{q}\kq \to \Sigma^{1,0}\s_{q+1}\kq$ denote the first slice differential as a map of motivic spectra, 
and similarly for $\KW_{\geq 0}$. 
By Theorem~\ref{thm:slices-kq}, 
$\sliced_{1}^{\kq}(q)$ is a map between finite sums of motivic Eilenberg-MacLane spectra for the groups $\ZZ$ and $\ZZ/2$. 
Thus $\sliced_{1}^{\kq}(q)$ can be described via its restriction $\sliced_{1}^{\kq}(q,i)$ to the summand corresponding to the unique suspension $\Sigma^{q+i,q}$.
We note that $\sliced_{1}^{\kq}(q,i)$ splits into at most three nontrivial components.

Let $\tau\in h^{0,1}\cong\mu_{2}(F)$ and $\rho\in h^{1,1}\cong F^{\times}/2$ denote the classes represented by $-1\in F$; 
$h^{p,q}$ is shorthand for the mod-$2$ motivic cohomology group of $F$ in degree $p$ and weight $q$.
There are canonical maps $\pr\colon\M\ZZ\to\M\ZZ/2$ and $\partial\colon\M\ZZ/2\to\Sigma^{1,0}\M\ZZ$.

\begin{theorem}
\label{thm:diff-kq}
When $\Char(F)\neq 2$ the $\sliced_{1}$-differential in the slice spectral sequence for $\kq$ is given by
\begin{align*}
\sliced_{1}^{\kq}(q,i) 
& =  
\begin{cases}
(\Sq^{3}\Sq^{1},\Sq^{2},0) &  q-1> i\equiv 0\bmod 4 \\
(\Sq^{3}\Sq^{1},\Sq^{2}+\rho\Sq^{1},\tau) &  q-1> i\equiv 2\bmod 4 \\
\end{cases} \\
\sliced_{1}^{\kq}(q,q) 
& =  
\begin{cases}
(0,\Sq^{2}\circ\pr,0) & q\equiv 0 \bmod 4 \\
(0,\Sq^{2}\circ \pr,\tau\circ \pr) & q\equiv 2 \bmod 4 \\
\end{cases} \\
\sliced_{1}^{\kq}(q,q-1) 
& =  
\begin{cases}
(\partial\Sq^{2}\Sq^{1}, \Sq^{2},0) & q\equiv 1 \bmod 4 \\
(\partial\Sq^{2}\Sq^{1}, \Sq^{2}+\rho\Sq^{1},\tau) & q\equiv 3 \bmod 4. 
\end{cases}  
\end{align*}
\end{theorem}
\begin{proof}
Use Theorem \ref{thm:slices-kq} and the identification of $\sliced_{1}^{\KQ}$ for $\KQ$ in \cite[Theorem 5.5]{ro.hermitian}.
\end{proof}

\begin{theorem}
\label{thm:diff-kw}
When $\Char(F)\neq 2$ the $\sliced_{1}$-differential in the slice spectral sequence for $\KW_{\geq 0}$ is given by
\begin{align*}
\sliced_{1}^{\KW_{\geq 0}}(q,i) 
& =  
\begin{cases} 
(\Sq^{3}\Sq^{1},\Sq^{2},0) & i \equiv 0\bmod 4 \\
(\Sq^{3}\Sq^{1},\Sq^{2}+\rho\Sq^{1},\tau) & i \equiv 2\bmod 4.
\end{cases} 
\end{align*}
\end{theorem}
\begin{proof}
This follows from Theorem \ref{thm:slices-kw} and the identification of $\sliced_{1}^{\KW}$ for $\KW$ recorded in \cite[Theorem 5.3]{ro.hermitian}.
\end{proof}

Following \cite[\S4]{ro.mult-slices-hermitian} we calculate the first slice differentials for $\kq$ and $\KW_{\geq 0}$ in terms of the multiplicative generators for their slices.

We note that $\dd^{\kq}_1(\sqrt{\alpha}^m{\eta}^n)$ is given by 
\begin{equation}
\label{eq:diff-kq}
\begin{cases} 
\tau\sqrt{\alpha}^{m-1}{\eta}^{n+3} + (\Sq^2+\rho \Sq^1) \sqrt{\alpha}^m{\eta}^{n+1} + \Sq^3\Sq^1 \sqrt{\alpha}^{m+1}{\eta}^{n-1} & m\equiv 1(2),n>1 \\ 
\Sq^2 \sqrt{\alpha}^m{\eta}^{n+1} + \Sq^3\Sq^1\sqrt{\alpha}^{m+1}{\eta}^{n-1} & m\equiv 0(2),n>1\\
\tau \sqrt{\alpha}^{m-1}{\eta}^{4} + (\Sq^2+\rho\Sq^1)\sqrt{\alpha}^m{\eta}^{2} + \delta\Sq^2\Sq^1 \sqrt{\alpha}^{m+1} & m\equiv 1(2),n=1 \\ 
\Sq^2 \sqrt{\alpha}^m{\eta}^{2} + \delta\Sq^2\Sq^1\sqrt{\alpha}^{m+1} & m\equiv 0(2),n=1\\
\tau\pr \sqrt{\alpha}^{m-1}{\eta}^{3} + \Sq^2\pr \sqrt{\alpha}^m{\eta}  & m\equiv 1(2),n=0 \\ 
\Sq^2 \sqrt{\alpha}^m{\eta} & m\equiv 0(2),n=0,
\end{cases} 
\end{equation}
while $\dd^{\KW_{\geq 0}}_1(\sqrt{\alpha}^m{\eta}^n)$ is given by
\begin{equation}
\label{eq:diff-kw}
\begin{cases} 
\tau \sqrt{\alpha}^{m-1}{\eta}^{n+3} + (\Sq^2+\rho \Sq^1) \sqrt{\alpha}^m{\eta}^{n+1} + \Sq^3\Sq^1 \sqrt{\alpha}^{m+1}{\eta}^{n-1}  & m\equiv 1(2)\\ 
\Sq^2 \sqrt{\alpha}^m{\eta}^{n+1} + \Sq^3\Sq^1 \sqrt{\alpha}^{m+1}{\eta}^{n-1} & m\equiv 0(2).  \\
\end{cases}
\end{equation}

\begin{remark}
The corresponding formula for $\dd^{\KQ}_1(\sqrt{\alpha}^m{\eta}^n)$ in \cite[\S4]{ro.mult-slices-hermitian} contains a typo when $m\equiv 1(2),n=0$.
We thank Bert Guillou for pointing this out to us.
\end{remark}

\begin{remark}
\label{remark:very-effective-kq-kw}
Bachmann \cite{bachmann.very} determined the very effective slices of $\KQ$ and hence of $\kq$ up to extensions. 
Additional work is needed to identify the corresponding first very effective slice differentials. 
A first step is to calculate the endomorphisms of the very effective zero slice of $\KQ$.
The very effective slices of $\KW_{\geq 0}$ were determined up to extensions in \cite[Lemma 6]{bachmann.very}.
\end{remark}

\section{Homotopy computations}
\label{sec:homotopy-computations}

First we identify the target of the slice spectral sequences for the sphere and very effective hermitian $K$-theory.
\begin{theorem}
\label{theorem:ccsssforspherekgl}
Over a field $F$ of characteristic $\neq 2$ there are conditionally convergent slice spectral sequences
\begin{equation}
\label{equation:splicespectralsequencesphere}
\pi_{\star}\s_{\ast}\unit
\Longrightarrow
\pi_{\star}\unit^{\wedge}_{\eta},
\end{equation}
and 
\begin{equation}
\label{equation:splicespectralsequencekq}
\pi_{\star}\s_{\ast}\kq
\Longrightarrow
\pi_{\star}\kq^{\wedge}_{\eta}.
\end{equation}
\end{theorem}
\begin{proof}
Here \eqref{equation:splicespectralsequencesphere} is shown in \cite[\S3]{rso.first}.
The only issue in \eqref{equation:splicespectralsequencekq} is to identify the quotient of $\kq$ by $\eta$ with a slice complete spectrum \cite[\S4]{rso.hlp}.
This follows directly from Lemma \ref{lem:kgl-effective-perfect}, Proposition \ref{prop:wood} and \cite[Lemma 3.11]{rso.hlp}.
\end{proof}

To formulate our identification of the $0$-line of $\kq$ we recall the definition of Milnor-Witt $K$-theory $K^{MW}_{*}(F)$ in \cite{morel.field}.
It is the quotient of the free associative integrally graded ring on the set of symbols $[F^\times] := \{[u]\mid u\in F^\times\}$ in degree $1$ and $\eta$ in degree $-1$ by the homogeneous ideal enforcing the relations
\begin{itemize}
\item[(1)] $[uv] = [u]+[v]+\eta[u][v]$ ($\eta$-twisted logarithm),
\item[(2)] $[u][v] = 0$ for $u+v=1$ (Steinberg relation),
\item[(3)] $[u]\eta=\eta[u]$ (commutativity), and
\item[(4)] $(2+[-1]\eta)\eta = 0$ (hyperbolic relation).
\end{itemize}
Milnor-Witt $K$-theory is $\varepsilon$-commutative for $\varepsilon = -(1+[-1]\eta)$. 
By work of Morel \cite{morelmotivicpi0} there is an isomorphism with the graded ring of endomorphisms of the sphere
\[
K^{MW}_{*}(F)
\cong 
\bigoplus_{n\in\ZZ}\pi_{n,n}\unit.
\]

Moreover, 
$K^{MW}_{0}(F) \cong GW(F)$, the Grothendieck-Witt ring of quadratic forms with its standard presentation,
inverting $\eta$ in $K^{MW}_{*}(F)$ yields the ring of Laurent polynomials $W(F)[\eta^{\pm 1}]$ over the Witt ring, 
and $K^{MW}_{*}(F)/\eta = K^M_*(F)$, the Milnor $K$-theory ring of $F$.

\begin{theorem}
\label{theorem:zeroline}
Over a field $F$ of characteristic $\neq 2$ the unit map $\unit\to\kq$ induces an isomorphism on $0$-lines
\begin{equation}
\label{equation:zerolines}
K^{MW}_{*}(F)
\overset{\cong}{\to}
\bigoplus_{n\in\ZZ}\pi_{n,n}\kq.
\end{equation}
\end{theorem}
\begin{proof}
Recall from \cite[\S5]{rso.first} the short exact sequence
\begin{equation}
\label{equation:ses1}
0 
\to \pi_{n,n}{\unit} 
\to \pi_{n,n}{\unit}^\wedge_\eta\directsum \pi_{n,n}{\unit}[\tfrac{1}{\eta}] 
\to \pi_{n,n}{\unit}^\wedge_\eta[\tfrac{1}{\eta}]\to 0.
\end{equation}
Similarly, 
following \cite[\S3]{rso.first}, 
the $\eta$-arithmetic square 
\[ 
\xymatrix{ 
\kq \ar[r] \ar[d] & \kq[\frac{1}{\eta}] \ar[d]  \\
{\kq}^\wedge_\eta \ar[r] & {\kq}^\wedge_\eta[\frac{1}{\eta}]   
}
\]
for very effective $K$-theory yields a short exact sequence
\begin{equation}
\label{equation:ses2}
0 
\to \pi_{n,n}{\kq} 
\to \pi_{n,n}{\kq}^\wedge_\eta\directsum \pi_{n,n}{\kq}[\tfrac{1}{\eta}] 
\to \pi_{n,n}{\kq}^\wedge_\eta[\tfrac{1}{\eta}]\to 0.
\end{equation}
Here we use the vanishing of $\pi_{n+1,n}{\kq}^\wedge_\eta[\tfrac{1}{\eta}]$ and $\pi_{n-1,n}{\kq}$.
On the terms contributing to the $0$-line, 
the map from \eqref{equation:splicespectralsequencesphere} to \eqref{equation:splicespectralsequencekq} is an isomorphism.
Theorem \ref{thm:diff-kq} combined with the same computations as in \cite[\S4]{rso.first} show the said isomorphism persists to the $E^{\infty}$-page.
By invoking Theorem \ref{theorem:ccsssforspherekgl} we conclude $\pi_{n,n}{\unit}^\wedge_\eta\overset{\cong}{\to} \pi_{n,n}{\kq}^\wedge_\eta$ and 
$\pi_{n,n}{\unit}^\wedge_\eta[\tfrac{1}{\eta}]\overset{\cong}{\to}\pi_{n,n}{\kq}^\wedge_\eta[\tfrac{1}{\eta}]$.
As noted above,
by \cite{morelmotivicpi0} we have $\pi_{n,n}{\unit}[\tfrac{1}{\eta}]\overset{\cong}{\to} \pi_{n,n}{\kq}[\tfrac{1}{\eta}]\cong \pi_{n,n}{\KW_{\geq 0}} \cong W(F)$.
A straightforward comparison between \eqref{equation:ses1} and \eqref{equation:ses2} allows us to deduce \eqref{equation:zerolines}.
\end{proof}

\begin{remark}
It was pointed to us by Bachmann that the results of \cite{bachmann.very} yield an isomorphism of the zeroth generalized slices $\tilde{\s}_0\unit\cong \tilde{\s}_0 \KQ$. This gives another proof for Theorem~\ref{theorem:zeroline}.
\end{remark}

We note the isomorphism $\pi_{n+1,n}\kq\overset{\cong}{\to}\pi_{n+1,n}\kq^\wedge_\eta$ follows as in \cite[Proposition 5.3]{rso.first}.
Thus for the purpose of identifying the $1$-line of $\kq$ we may use Theorem \ref{thm:diff-kq} and computations as in \cite[\S4]{rso.first} to deduce:
\begin{proposition}
\label{prop:Einfty-kq}
The only nontrivial terms in \eqref{equation:splicespectralsequencekq} contributing to $\pi_{n+1,n}\kq$ are 
\[ 
E^\infty_{n+1,q,n}(\kq) 
= 
\begin{cases}
h^{-n+1,-n+2}/\Sq^{2}(h^{-n-1,-n+1}) & q=2 \\
h^{-n,-n+1}/\Sq^{2}\pr(H^{-n-2,-n}) & q=1 \\
H^{-n-1,-n} & q=0.
\end{cases}
\]
\end{proposition}
Here $h^{i,j}$ and $H^{i,j}$ denote the mod-$2$ and integral motivic cohomology groups of $F$ in degree $i$ and weight $j$.
This determines the $1$-line of $\kq$ up to extensions.
When $n>1$ we read off the vanishing $\pi_{n+1,n}\kq=0$.
The first nontrivial group on the $1$-line is $\pi_{2,1}\kq\cong\mu_{2}(F)\cong\ZZ/2$.
When $n=0$ we obtain $\pi_{1,0}\kq\cong \pi_{1,0}\KQ\cong F^{\times}/2\oplus\mu_{2}(F)$.
Furthermore,
there is a short exact sequence
\begin{equation}
0 
\to h^{2,3}/\Sq^{2}(h^{0,2})
\to \pi_{0,-1}\kq 
\to h^{1,2}
\to 
0.
\end{equation}
When $n\leq -2$ the group $\pi_{n+1,n}\kq$ surjects onto the integral motivic cohomology group $H^{-n-1,-n}$,
with kernel described by Proposition \ref{prop:Einfty-kq}.
\vspace{0.1in}

{\bf Acknowledgements.}
Mark Behrens asked about a motivic version of $\ko$ during Bert Guillou's {\tt ECHT} talk in Spring 2017. 
His question prompted a first version of this paper.
We also acknowledge Tom Bachmann's related work in \cite{bachmann.very}.
Theorem~\ref{thm:slices-kq} was first obtained in a different way during the first author's visit to Universit\"at Osnabr\"uck in November 2016.
Work on this paper took place at Institut Mittag-Leffler in Spring 2017,
where the first author held a postdoctoral fellowship financed by Vergstiftelsen, 
and the Hausdorff Research Institute for Mathematics in Summer 2017. 
We thank both institutions for excellent working conditions and support.
The authors acknowledge support from both the DFG priority programme ``Homotopy theory and algebraic geometry'' and the RCN programme ``Motivic Hopf equations''. 
Ananyevskiy is supported by RFBR grants 15-01-03034 and 16-01-00750, and by ``Native towns", a social investment program of PJSC ``Gazprom Neft".
{\O}stv{\ae}r is supported by a Friedrich Wilhelm Bessel Research Award from the Humboldt Foundation.

\bibliographystyle{plain}
\bibliography{sphereslices}

\vspace{0.1in}

\begin{center}
St.~Petersburg Department, Steklov Math.~Institute, and \\ Chebyshev Laboratory, St.~Petersburg State University, Russia. \\
e-mail: alseang@gmail.com
\end{center}
\begin{center}
Institut f\"ur Mathematik, Universit\"at Osnabr\"uck, Germany.\\
e-mail: oliver.roendigs@uni-osnabrueck.de
\end{center}
\begin{center}
Department of Mathematics, University of Oslo, Norway.\\
e-mail: paularne@math.uio.no
\end{center}

\end{document}